\documentclass[12pt,twoside]{article}
\usepackage[english]{babel}
\usepackage{amssymb}
\usepackage{amsmath}
\usepackage{latexsym}
\usepackage{a4,euscript,color}

\usepackage{amsthm}

\usepackage{hyperref}

\usepackage{comment}

\newenvironment{auskommentiert}{}{}
\excludecomment{auskommentiert}

\usepackage[arrow, curve, matrix]{xy}

\newtheorem{theo}{Theorem}[section]
\newtheorem{prop}[theo]{Proposition}
\newtheorem{cor}[theo]{Corollary}
\newtheorem{lem}[theo]{Lemma}

\newenvironment{remark}{{\noindent \bf Remark \refstepcounter{theo}\thetheo} }{ \vspace{1ex}}

\newcount\colveccount
\newcommand*\colvec[1]{
        \global\colveccount#1
        \begin{pmatrix}
        \colvecnext
}
\def\colvecnext#1{
        #1
        \global\advance\colveccount-1
        \ifnum\colveccount>0
                \\
                \expandafter\colvecnext
        \else
                \end{pmatrix}
        \fi
}

\newcommand{\ul}{\underline}
\newcommand{\ol}{\overline}
\newcommand{\nats}{I\hspace{-0,3em}N }

\newcommand{\Pmw}[1]{{\mathfrak P}(#1)}
\newcommand{\Pm}[1]{{\mathfrak P}_0(#1)}
\newcommand{\gx}{{\mathfrak X}}
\newcommand{\ga}{{\mathfrak A}}

\newcommand{\gb}{{\mathfrak B}}

\newcommand{\gm}{{\mathfrak M}}

\newcommand{\gc}{{\mathfrak C}}
\newcommand{\gd}{{\mathfrak D}}

\newcommand{\go}{{\mathfrak O}}

\newcommand{\gs}{{\mathfrak S}}

\newcommand{\Uf}[1]{{\mathfrak F}_0(#1{})}
\newcommand{\Fi}[1]{{\mathfrak F}(#1{})}

\newcommand{\epf}[1]{\overset{\bullet}{#1{}}}

\newcommand{\ca}{{\cal A}}
\newcommand{\cb}{{\cal B}}

\newcommand{\cf}{{\cal F}}

\newcommand{\ch}{{\cal H}}

\newcommand{\cu}{{\cal U}}

\newcommand{\ufi}{\underline{\varphi}}
\newcommand{\ups}{\underline{\psi}}
\renewcommand{\phi}{\varphi}
\newcommand{\eps}{\varepsilon}

\renewcommand{\labelenumii}{(\alph{enumii})}
\renewcommand{\labelenumi}{(\arabic{enumi}) }

\newcommand{\bdm}{\begin{displaymath}}
\newcommand{\edm}{\end{displaymath}}

\newcommand{\vietoris}[2]{\langle #1\rangle_{_{#2}}}

\begin{document}

\def\timestring{\begingroup
   \count0 = \time
   \divide\count0 by 60
   \count2 = \count0   
   \count4 = \time
   \multiply\count0 by 60
   \advance\count4 by -\count0   
   \ifnum\count4<10
      \toks1 = {0}%
   \else
      \toks1 = {}%
   \fi
   \ifnum\count2<10
      \toks0 = {0}%
   \else
      \toks0 = {}%
   \fi
   \ifnum\count2=0
      \toks0 = {00}
   \fi
  \the\toks0 \number\count2:\the\toks1 \number\count4 \thinspace
\endgroup}%
\newcommand{\timestamp}{\today, \timestring}

\hyphenation{na-tu-ral in-ner-to-po-lo-gi-cal Rau-mes O-ber-ul-tra-fil-ter po-wer-fil-ter
po-wer-fil-ter-spa-ces theo-rem mul-ti-fil-ter Mul-ti-fil-ter Ei-gen-schaft to-po-lo-gisch-es
gleich-ste-tig com-pac-ti-fi-ca-tion}

\renewcommand{\labelenumii}{(\alph{enumii})}
\renewcommand{\labelenumi}{(\arabic{enumi}) }

\author{Ren\'e Bartsch}

\title{Hyperspaces in topological Categories}

\maketitle

\abstract{Hyperspaces form a powerful tool in some branches of mathematics: lots of fractal and other geometric objects can be viewed as fixed points of some functions in suitable hyperspaces - as well as interesting classes of formal languages in theoretical computer sciences, for example (to illustrate the wide scope of this concept). Moreover, there are many connections between hyperspaces and function spaces in topology. Thus results from hyperspaces help to get new results in function spaces and vice versa.\\

Unfortunately, there ist no natural hyperspace construction known for general topological categories (in contrast to the situation for function spaces).\\

We will shortly present a rather combinatorial idea for the transfer of structure from a set $X$ to a subset of $\mathfrak{P}(X)$, just to motivate an interesting question in set theory.\\

Then we will propose and discuss a new approach to define hyperstructures, which works in every cartesian closed topological category, and so applies to every topological category, using it's topological universe hull.}

\section{Preliminary Def\/initions and Results}

For a given set $X$ we denote by $\Pmw{X}$ the power set of $X$, by $\Pm{X}$ the power set without the empty set. By $\Fi{X}$ (resp. $\Uf{X}$) we mean the set of all filters (resp. ultrafilters) on $X$; if $\phi$ is a filter on $X$, the term $\Uf{\phi}$ denotes the set of all ultrafilters on $X$, which contain $\phi$. For $x\in X$ we denote by $\epf{x}:=\{A\subseteq X|\; x\in A\}$ the singleton filter on $X$, generated by $\{x\}$. With $\gs(X):=\{\epf{x}|\; x\in X\}$ we mean the family of all singleton filters on $X$. By $\ca(X):=\{f: \Pm{X}\to X|\; \forall A\in \Pm{X}: f(A)\in A\}$ we denote the set of choice functions over $X$.

For a topological space $(X,\tau)$ we denote by $q_{\tau}:=\{(\phi,x)\in\Fi{X}\times X|\; \phi\supseteq \epf{x}\cap\tau\}$ the convergence induced by $\tau$.\\

For families $\ga\subseteq\Pmw{X}$ and any $M\subseteq X$ we set 
\bdm
M^{-_{\ga}}:=\{A\in\ga|\; A\cap M\not = \emptyset\}
\edm
and 
\bdm
M^{+_{\ga}}:=\{A\in\ga|\; A\cap M=\emptyset\}\text{ .}
\edm
Then for a topological space $(X,\tau)$ on $\ga\subseteq \Pm{X}$ the {\em lower Vietoris topology} $\tau_{l,\ga}$ is defined by the subbase $\{O^{-_{\ga}}|\; O\in\tau\}$, whereas the {\em upper Vietoris topology} $\tau_{u,\ga}$ on $\ga$ comes from the subbase $\{(X\setminus O)^{+_{\ga}}|\; O\in \tau\}$. The {\em Vietoris topology} on $\ga$ is $\tau_{V,\ga}:=\tau_{l,\ga}\vee\tau_{u,\ga}$. In most cases $\ga$ is chosen as the family $Cl(X)$ of the closed, or $K(X)$ of the compact subsets of a topological space $(X,\tau)$, or as the entire $\Pm{X}$. Whenever there is no doubt about $\ga$, we will omit it as sub- and superscript.\\

We will need some basic facts about Stone-\v{C}ech-compactification of discrete spaces.\\

A discrete space $(D,\delta)$ clearly is $T_4$ and Hausdorff, so its Stone-\v{C}ech-compactification is homeomorphic to its Wallman extension, consisting in this case just of the set $\Uf{D}$, where the singleton filters are identified with their generating points via $w:D\to \Uf{D}:w(x):=\epf{x}$, endowed with the topology generated from the base consisting of all sets $\Uf{M}$, with $M\subseteq D$ (see \cite{engelking}, p.176ff).\\

\begin{prop}
Let $(D,\delta)$ be a discrete topological space. Then for its Stone-\v{C}ech-compactification $(B,\sigma)$ hold
\begin{enumerate}
\renewcommand{\labelenumi}{(\alph{enumi})}
\renewcommand{\theenumi}{(\alph{enumi})}
\item\label{preprop-1} For all $M\subseteq D$ in $B$ the closure $\ol{M}$ is clopen.
\item\label{preprop-2} $\sigma$ has a base consisting of clopen sets.
\item\label{preprop-3} All clopen sets $C$ in $(B,\sigma)$ are of the form $C=\ol{C\cap D}$.
\end{enumerate}
\end{prop}

\begin{proof}
We use the homeomorphy of $(B,\sigma)$ to the Wallman extension.\\

\ref{preprop-1} We have $\ol{w(M)}=\Uf{M}$: from $\Uf{M}=\Uf{D}\setminus \Uf{D\setminus M}$ we conclude, that $\Uf{M}$ is closed and of course it contains $w(M)$. So, $\ol{w(M)}\subseteq\Uf{M}$ follows. If there would be a filter $\phi\in \Uf{M}$ which belongs not to $\ol{w(M)}$, then there would exist a base set $\Uf{S}$, $S\subseteq D$, of $\sigma$ s.t. $\phi\in \Uf{S}$ and $\Uf{S}\cap w(M)=\emptyset$. But this implies $M\cap S=\emptyset$, and thus $\Uf{S}\cap\Uf{M}=\emptyset$ - in contradiction to $\phi\in \Uf{S}\cap\Uf{M}$. So, we have indeed $\ol{w(M)}=\Uf{M}$, which is also open, because it belongs to our defining base of $\sigma$.\\

\ref{preprop-2} Follows immediately from \ref{preprop-1}.\\

\ref{preprop-3} Let $C\subseteq B$ be clopen. Then for all $c\in C$ there exists a basic open set $\Uf{M_c}$ with $M_c\subseteq D$, s.t. $c\in \Uf{M_c}\subseteq C$, because $C$ is open. From closedness of $C$ automatically follows compactness, because $(B,\sigma)$ is compact, thus there are finitely many 
$M_{c_1},...,M_{c_n}$ with $C=\bigcup_{i=1}^n\Uf{M_{c_i}}$. Now, for such finite union we have generally $\bigcup_{i=1}^n\Uf{M_{c_i}}=\Uf{\bigcup_{i=1}^n M_{c_i}}$ and it is clear, that $w(\bigcup_{i=1}^n M_{c_i})= C\cap w(D)$ holds.
\end{proof}

\section{Choice Functions}

\begin{lem}
Let $(X,\tau)$ be a topological space, $\widehat{\phi}$ an ultrafilter on $\Pm{X}$ and let $P:=\{p\in X|\; \exists f\in \ca(X): f(\widehat{\phi})\overset{\tau}{\rightarrow}p\}$. Then $\widehat{\phi}$ converges to the $\tau$-closure $\overline{P}\in \Pmw{X}$ w.r.t. $\tau_l$.
\end{lem}
\begin{proof}
Let $O\in\tau$ be given with $\overline{P}\in O^-$, i.e. $\exists x\in O\cap\overline{P}$ and therefore $\exists p\in O\cap P$. So, there exists a choice function $f\in \ca(X)$ s.t. $f(\widehat{\phi})\to p$, implying $\exists \gc\in\widehat{\phi}: f(\gc)\subseteq O$, thus $\forall C\in\gc: C\cap O\neq \emptyset$. So, we have $\gc\subseteq O^-\in \widehat{\phi}$.
\end{proof}

Obviously, every filter converging to a subset $B\subseteq X$ w.r.t. $\tau_l$, converges to every $A\subseteq B$, too.

\begin{prop}
Let $(X,\tau)$ be a locally compact topological space, $P:=\{p\in X|\; \exists f\in \ca(X): f(\widehat{\phi})\overset{\tau}{\rightarrow}p\}$ and let 
$\widehat{\phi}$ be an ultrafilter on $\Pm{X}$ with $\widehat{\phi}\overset{\tau_l}{\rightarrow}A\in\Pmw{X}$. Then $A\subseteq \overline{P}$ holds.
\end{prop}
\begin{proof}
For every $a\in A$ and every open neigbourhood $U\in \epf{a}\cap\tau$ there is a compact $K\subseteq X$ and $O\in \tau$ with $a\in O\subseteq K\subseteq U$. Now, $A\in O^-$ and consequently $O^-\in \widehat{\phi}$. There exists a choice function $f\in \ca(X)$ s.t. $\forall M\in O^-: f(M)\in O$, yielding $f(O^-)\subseteq O$. Thus we find $O\subseteq K\in f(\widehat{\phi})$. But $\widehat{\phi}$ is an ultrafilter and so $f(\widehat{\phi})$ is.
Consequently, by the compactness of $K$, $f(\widehat{\phi})$ converges to some $k\in K\subseteq U$ and we have now $k\in U\cap P\neq\emptyset$. This yields $a\in \overline{P}$. 
\end{proof}

So, for locally compact spaces $(X,\tau)$, every ultrafilter on $\Pmw{X}$  converges w.r.t. $\tau_l$ exactly to the subsets of its corresponding $\overline{P}$.\\

In general, it seems not so easy to guarantee that a single choice function yields a convergent image near a point of the limit of a $\tau_l$-convergent ultrafilter. To give us some more latitude, we try now filters on $\ca(X)$ instead.

\begin{prop}
Let $(X,\tau)$ be a nested neighbourhood space, let $\widehat{\phi}$ be an ultrafilter on $\Pm{X}$ with $\widehat{\phi}\overset{\tau_l}{\rightarrow}A\in\Pmw{X}$ and let 

$P:=\{p\in X|\; \exists \cf\in \Fi{\ca(X)}: \cf(\widehat{\phi})\overset{\tau}{\rightarrow}p\}$. 

\noindent Then $A\subseteq P$ holds.
\end{prop}
\begin{proof}
Let $a\in A$ be given. Let $\gb\subseteq\tau\cap\epf{a}$ be a neighbourhood base of $a$ which is totally ordered by inclusion.
For every $U\in\gb$ let $\ca_U:=\{f\in\ca(X)|\; \forall H\in U^-: f(H)\in U\}$. 
For $n\in\nats$ and $U_1,...,U_n\in\gb$ with $U_1\subseteq \cdots \subseteq U_n$ we have $U_1^-\subseteq\cdots \subseteq U_n^-$. Then for every $H\in U_n^-$ there is a unique $k(H):=\min\{i\in\nats|\; H\cap U_i\neq\emptyset\}$. Now, we see $\ca_{U_1,...,U_n}:=\{f\in\ca(X)|\; \forall H\in U_n^-: f(H)\in U_{k(H)}\}\neq\emptyset$ and $\ca_{U_1,...,U_n}\subseteq \bigcap_{i=1}^n\ca_{U_i}$. So, the family $\{\ca_{U}|\; U\in\gb\}$ forms a subbase for a filter $\cf$ on $\ca$ and obviously $\ca_U(U^-)\subseteq U$ holds. From $\widehat{\phi}\overset{\tau_l}{\rightarrow}A\ni a$ follows $\forall U\in\gb: U^-\in\widehat{\phi}$ and so we find $\forall U\in\gb: U\in \cf(\widehat{\phi})$, yielding $\cf(\widehat{\phi})\to a$. 
\end{proof}

\begin{auskommentiert}

\begin{theo}
\label{unif}
Let $(X,\cu)$ be a uniform space, $\tau_{\cu}$ the induced topology on $X$, $\gx$ of compact subsets of $X$ and $\hat{\cu}:=[\{\hat{R}|R\in \cu\}]$ with
$\hat{R}:=\{(A,B)\in \gx\times\gx|R(A)\supseteq B\wedge R^{-1}(B)\supseteq A\}$ the corresponding
Hausdorff--uniformity on $\gx$. If $\ufi\in \Fi{\gx}$, then the following are equivalent:
\begin{enumerate}
\item $\ufi \overset{\tau_{\hat{\cu}}}{\longrightarrow} A\in \gx$,
\item \begin{enumerate} \item $\forall f\in \ca(X),\ups\in
\Uf{\ufi}:\exists a\in A: f(\ups)\overset{\tau_{\cu}}{\longrightarrow} a$ and
\item $\forall a\in A:\exists f\in \ca(X): f(\ufi)\overset{\tau_{\cu}}{\longrightarrow}
a$.
\end{enumerate}
\end{enumerate}
\end{theo}

\begin{proof}
$(2)\Rightarrow (1)$: Angenommen, es existiert $R\in \cu$
mit $\hat{R}(A)\not\in \ufi$. Dann existiert auch
ein $\ups\in \Phi_0(\ufi)$ mit $\hat{R}(A)\not\in
\ups$, d.h. $\hat{R}(A)^c=\{M\in \gx|R(A)\not\supseteq M \vee R^{-1}(M)\not\supseteq A \}
\in \ups$.
Daraus folgt ${\frak M}_1:=\{ M\in \gx|R(A)\not\supseteq M \} \in \ups$
oder ${\frak M}_2:=\{M\in \gx| R^{-1}(M)\not\supseteq A \}\in
\ups$.\\
Im Falle ${\frak M}_1\in \ups$ haben wir $\forall M\in {\frak M}_1: \exists \gamma(M)\in M:\gamma(M)\not\in
R(A)$, wir finden also ein $\gamma\in \ca$ mit $\gamma({\frak M}_1)\subseteq
R(A)^c$, also $R(A)^c\in \gamma(\ups)$. Daraus folgt $\forall a\in A: R(a)\not\in \gamma(\ups)$
und somit $\forall a\in A:\gamma(\ups)\not\rightarrow a$ im
Widerspruch zu 2(a).\\
Im Falle ${\frak M}_2\in \ups$ finden wir $\forall M\in {\frak M}_2:\exists \alpha(M)\in A: \alpha(M)\not\in
R^{-1}(M)$. Damit bilden wir $\alpha(\ups)\in \Phi_0(A)$ und
finden wegen der Kompaktheit von $A$ ein $a_0\in A$ mit $\alpha(\ups)\rightarrow
a_0$, also f\"ur $S=S^{-1}\in \cu$ mit $S\circ S\subseteq R$ auch $S^{-1}(a_0)=S(a_0)\in
\alpha(\ups)$. Somit gilt $\alpha^{-1}(S^{-1}(a_0))\in \ups$, also
auch ${\frak M}_2\cap \alpha^{-1}(S^{-1}(a_0))\in \ups$. Nun gilt $\forall M\in {\frak M}_2: \forall \beta\in \ca: (\alpha(M),\beta(M)\not\in
R$, ferner $\forall M\in \alpha^{-1}(S^{-1}(a_0)): (\alpha(M),a_0)\in
S$, was f"ur alle $M$ aus ${\frak M}_2\cap \alpha^{-1}(S^{-1}(a_0))$
und alle $\beta\in \ca$ wegen $S\circ S\subseteq R$ sofort $(a_0, \beta(M))\not\in S$
erzwingt. Das bedeutet aber $\forall \beta\in \ca: S(a_0)\not\in
\beta(\ups)$, also $\beta(\ups)\not\rightarrow a_0$, mithin $\beta(\ufi)\not\rightarrow a_0$
im Widerspruch zu 2(b). \\
Die Annahme eines $R\in \cu$ mit $\hat{R}(A)\not\in \ufi$ f\"uhrt
also stets zum Widerspruch, so da\ss\ $\ufi \overset{\tau_{\hat{\cu}}}{\longrightarrow} A$
gelten mu\ss.\\

$(1)\Rightarrow (2)$: Seien $f\in \ca$ und $\ups\in \Phi_0(\ufi)$
gegeben. Aus (1) folgt zun\\"achst
$\ups\overset{h_d}{\rightarrow}A$. Angenommen, es gilt nicht 2(a).
Dann gilt $\forall a\in A:\ul{U}^d(a)\not\subseteq f(\ups)$, d.h. $\forall a\in A:\exists
\eps_a>0: U_{\eps_a}(a)\not\in f(\ups)$, also jeweils $U_{\eps_a}(a)^c\in
f(\ups)$. Nun gilt $V:=\bigcup_{a\in A}U_{\eps_a}(a)\supseteq
A$, so da"s wegen der Kompaktheit von $A$ endlich viele $a_1,...,a_n\in A$
existieren mit $\bigcup_{i=1}^nU_{\eps_{a_i}}(a_i)\supseteq
A$. Dazu ist $V^c=\bigcap_{i=1}^nU_{\eps_{a_i}}(a_i)^c\in
f(\ups)$, also $V\not\in f(\ups)$. Nach Proposition (\ref{eps1})
existiert nun ein $\eps_0>0$ mit $U_{\eps_0}^d(A)\subseteq
V$, also $U_{\eps_0}(A)^c\in f(\ups)$. Demnach ist ${\frak M}:=f^{-1}(U_{\eps_0}(A)^c)\in
\ups$, d.h.
\bdm
\begin{array}{lcl}
\forall M\in {\frak M}\in \ups&:& \exists m(=f(M))\in M: d(m,A)\geq \eps_0\\
\Longrightarrow &:&\sup_{m\in M}d(m,A)\geq \eps_0 \\
\Longrightarrow &:&h_d(M,A)\geq \eps_0 \\
\Rightarrow {\frak M}\subseteq
\left(U_{\eps_0}^{h_d}(A)\right)^c\mbox{ ,}&&
\end{array}
\edm
also $U_{\eps_0}^{h_d}(A)\not\in \ups$ im Widerspruch zur
Konvergenz von $\ups$.\\
Sei nun $a_0\in A$ gegeben und wiederum (1) vorausgesetzt. F"ur
beliebiges $M\subseteq X$ gilt $h_d(A,M)\geq \sup_{a\in A}d(a,M)\geq d(a_0, M)$
und f"ur kompaktes $M$ ferner $d(a_0,M)=\inf_{m\in M}d(a_0,m)=\min_{m\in
M}d(a_0,m)$. Somit finden wir zu jedem $M\in \gx$ ein $\beta(M)\in M$
mit $d(a_0,\beta(M))=d(a_0,M)\leq d(A,M)$. Damit folgt aber f"ur
die Funktion $\beta\in \ca$ unmittelbar $\beta(U_{\eps}^{h_d}(A))\subseteq
U_{\eps}^d(a_0)$, so da"s aus $\ul{U}^{h_d}(A)\subseteq \ufi$
sogleich $\ul{U}^d(a_0)\subseteq \beta(\ufi)$ folgt, womit wir
2(a) erhalten. 
\end{proof}

\end{auskommentiert}

Concerning the idea to describe hyperstructures via choice functions, a simple question arises immediately: if we start, as very simple case, with a discrete space and require an extremely hard condition about choice functions - will we get then a discrete hyperspace?\\

This simple question leads at once to much too hard problems in our set theoretical background, as we will see now.\\

Let $X$ be a set and denote by $\Pm{X}:=\{M\subseteq X|\; M\neq\emptyset\}$
the set of all nonempty subsets of $X$.
For $x\in X$ let $\epf{x}:=\{A\subseteq X|\; x\in A\}$ the singleton filter generated by $\{x\}$.\\
Further let $\ca:=\{f\in X^{\Pm{X}}|\; \forall A\in \Pm{X}: f(A)\in A\}$
the set of such functions from $\Pm{X}$ to $X$, which assign to every nonempty subset $A$ of $X$ an element of $A$.\\
Now, let be given a filter $\widehat{\phi}$ on $\Pm{X}$ with the property\\
$\forall f\in \ca: \exists x_f\in X: f(\widehat{\phi})=\epf{x_f}$.\\

We claim, that $\widehat{\phi}$ is an ultrafilter on $\Pm{X}$.\\

\begin{proof}
We show, that for every subset of $\Pm{X}$ either this subset itself or its complement
belongs to $\widehat{\phi}$.\\
{\bf Case 1: } $\exists a\in X: \forall f\in \ca:
f(\widehat{\phi})=\epf{a}$\\
This implies, that even a choice function, that avoids the point $a$ wherever possible, would map our $\widehat{\phi}$ to $\epf{a}$. Then $\widehat{\phi}$ contains $\{\{a\}\}$, and so it is a singleton filter, thus an ultrafilter.\\
{\bf Case 2: }
 $\exists a,b\in X, f,g\in \ca: a\neq b\wedge f(\widehat{\phi})=\epf{a}\wedge g(\widehat{\phi})=\epf{b}$.\\
It follows $f^{-1}(\{a\})\subseteq\{A\in \Pm{X}|\; a\in A\}=:\ga\in \widehat{\phi}$
and $g^{-1}(\{b\})\subseteq\{A\in \Pm{X}|\; b\in A\}=:\gb\in \widehat{\phi}$, thus
$\gc:=\ga\cap\gb\in \widehat{\phi}$. So, it suffices to show, that for every subset $\gd$ of $\gc$ either $\gd$ is an element of $\widehat{\phi}$, or its complement in $\gc$, $\gc\setminus \gd$. 

Assume, this would not hold, i.e.
\begin{align}
\exists \gd\subseteq \gc: \gd\not\in \widehat{\phi}\wedge {\gd}^c\not\in\widehat{\phi}
\end{align}
 (with $\gd^c:=\gc\setminus\gd$).
This implies 
\begin{align}\label{allkompatibel}
\forall \gm\in \widehat{\phi}:
\gm\cap\gd\neq\emptyset\wedge\gm\cap\gd^c\neq\emptyset\mbox{ .}
\end{align} 
Now, every element of $\gc$ (and so every element of  $\gd$ as well as of $\gd^c$ contains  $a$ {\em and} $b$. Consequently, there exists a choice function $h\in \ca$
s.t. \bdm
h(M):=\left\{\begin{array}{ccl}
a&;&\;\mbox{ if }M\in \gd\\
b&;&\;\mbox{ if } M\in \gd^c\\
f(M)&;&\;\mbox{ if } M\not\in \gc
\end{array}\right.
\edm
{\footnotesize (Rem.: The designation $h(M)=f(M)$ for $M\not\in \gc$ is not essential, we just have to make any choice~--~and $f(M)$ is anyhow possible, because of the existence of $f$.})

Because of \eqref{allkompatibel} it follows $h(\widehat{\phi}):=[\{a,b\}]$, which is not a singleton filter~--~in contradiction to our precondition.
\end{proof}

Now, the question struggles, wheither or not such a filter $\widehat{\phi}$ must be a singleton filter on $\Pm{X}$.\\

In order to characterize the Lindel\"of-property by convergence of filters, as proposed in \cite{froschbuch}, lemma 5.1.20, the notion of{\em \glqq countable completeness\grqq\ } was introduced for filters; it means, that the intersection of every at most countable family of elements of a filter is an element of the filter again.

Let $X$ be a set. 


We say, a filter $\Phi$ has the {\em property (A) w.r.t. $X$} iff 
$\Phi$ is a filter on $\Pm{X}$ and fulfills
\begin{align}\renewcommand{\theequation}{A}\label{propertyA}
\forall f\in \ca(X): \exists x_f\in X: f(\Phi)=\epf{x_f}\mbox{ .}
\end{align}

\begin{prop}
Every filter $\Phi$ with property (A) w.r.t. a set $X$  is countably complete.
\end{prop}

\begin{proof}
 Let's make a few observations at first:
 \begin{enumerate}\renewcommand{\labelenumi}{(\alph{enumi})}\renewcommand{\theenumi}{(\alph{enumi})}
\item As shown before, $\Phi$ is an ultrafilter.
\item If $\Phi$ has a countable subset, whose intersection does'nt belong to $\Phi$, then  $\Phi$ has also a countable subset with empty intersection.
\item Let $P:=\{n\in X|\; \exists f\in\ca(X): f(\Phi)=\epf{n}\}$. Obviously, $\Pm{P}\in\Phi$ holds - otherwise we would have $\Pm{X}\setminus\Pm{P}\in\Phi$ and thereon a choice function exists, which misses $P$ completely. If $P$ is finite, then $\Pm{P}$, too - and then $\Phi$ must be a singleton filter, which is trivially countably complete. 
\item For every $p\in P$ holds $\epf{p}\in\Phi$.
\item\label{obs-e} If $P$ is infinite, there exists a countably infinite subset $P':=\{p_n|\; n\in\nats\}$ of $P$, s.t. always $n\neq m\implies p_n\neq p_m$ holds. 
\end{enumerate}

So, let $P$ be infinite and let $P'$ be given as mentioned in \ref{obs-e}.\\ 

Assume, there exists a countable infinite subset $\ga:=\{A_n|\, n\in\nats\}$ of $\Phi$ with $\bigcap_{n\in\nats}A_n=\emptyset$.\\ 

In an usual manner, we construct a strictly decreasing sequence of elements of $\Phi$: we define at first $A_0':=A_0\cap\epf{p_0}$ continue inductively from given $A_0', ... , A_n'$ with $A_{n+1}':=(A_{n+1}\cap\epf{p_{n+1}})\cap\bigcap_{i=0}^nA_i'$.\\

Then the {\em set} $\ga':=\{A_n'|\; n\in\nats\}$ is again a subset of $\Phi$ ans also countably inifinite (otherwise the intersection about all of its elements would be nonempty, in contrast to being contained in the intersection of all $A_n$). Furthermore, the {\em sequence} $(A_n')_{n\in\nats}$ is decreasing and we have $\bigcap_{i=0}^nA_n'=\emptyset$. By removing of possible doublets we get a strongly decreasing sequence $(B_n)_{n\in\nats}$, which exhausts $\ga'$: we define $B_0:=A_0'$ (and $p_0':=p_0$) and then inductively from $B_0,...,B_i$ always $n_{i+1}:=min\{n\in\nats|\; A_n'\subsetneqq B_i\}$ and $B_{i+1}:=A_{n_{i+1}}'$ (and $p_{i+1}':=p_{n_{i+1}}$).\\

So we get $\bigcap_{n\in\nats}B_n\subseteq \bigcap_{n\in\nats}A_n=\emptyset$.\\

Consequently the mapping $\kappa:B_0\to \nats: \kappa(M):=\min\{n\in\nats|\; M\not\in B_n\}-1$ is well defined.\\

Now consider any $g\in\ca(X)$.\\ 
We define the choice function
\bdm
f:\Pm{X}\to X:f(M):=\left\{
\begin{array}{ccl}
p_{\kappa(M)}'&;& M\in B_0\\
g(M)&;& sonst
\end{array}\right.
\edm

We find for all $n\in\nats$: $f(B_n)=\{p_i'|\; i\geq n\}$, thus $f(\Phi)$ includes the cofinite filter on $P'':=\{p_i'|\;i\in\nats\}$ - so $\Phi$ cannot be a singleton filter, in contradiction to our precondition on $\Phi$.
\end{proof}

\begin{prop}
On $\Pm{\nats}$ there exists no countably complete free ultrafilter.
\end{prop}

\begin{proof}
Assume, $\Phi$ would be a free, countably complete ultrafilter on $\Pm{\nats}$. Let $h$ be any bijection from $\Pm{\nats}$ to the real interval $[0,1]$. Then $h(\Phi)$ is a free countably complete ultrafilter on $[0,1]$. As ultrafilter on this euclidian compact interval it converges to an element $r\in[0,1]$, i.e. $h(\Phi)\supseteq \{U_{\frac{1}{n}}(r)|\; n\in\nats, n>0\}$. Because $h(\Phi)$ is free, it holds $[0,1]\setminus\{r\}\in h(\Phi)$, so  $h(\Phi)\supseteq \ga:=\{U_{\frac{1}{n}}(r)\setminus\{r\}|\; n\in\nats, n>0\}$ follows. From countable completeness we get now $\emptyset=\bigcap_{M\in\ga}M\in h(\Phi)$ - in contradiction to the filter property.
\end{proof}

From this follows immediatly:
\begin{cor}
The filters having property (A) w.r.t. $\nats$ are exactly the singleton filters on $\Pm{\nats}$.
\end{cor}

{\bf Remark: } Unfortunately, at this point it's over with provable answers to the questions whether or not a filter with property (A) must be a singleton filter 9in our set theoretic realm).  The colleagues concerned with axiomatics, model theory etc. showed
\begin{enumerate}
\item Countable complete free ultrafilters exist (anywho), iff measurable cardinals exist.
\item Every measurable cardinal is inaccessible. 
\end{enumerate}
Okay, but now:
\begin{enumerate}
\setcounter{enumi}{2}
\item If we add to our set theoretic axiomatic (ZFC) the axiom \glqq There exists an inaccessible cardinal.\grqq, the new system does prove the consistency of ZFC.
\item Consistency of ZFC implies consistency of ZFC + "there are no inaccessible cardinals".
\end{enumerate}
The reader may find some useful explications and references here:\\ 
\href{http://en.wikipedia.org/wiki/Inaccessible_cardinal}{http://en.wikipedia.org/wiki/Inaccessible\_cardinal} \\


We find, unless some characterizations via choice functions may be sometimes useful, it seems to be not the best idea to use them for definition of hyperstructures in a general setting - just for instance, because they have a tendency to lead rapidly to highly inconvenient set theoretic trouble.\\

So, let's try another way.

\section{Vietoris Hyperstructure as final w.r.t. Function Spaces}

Remember a wide class of function space structures, defined for
$Y^X$ or $C(X,Y)$: the so called set--open topologies, examined in \cite{arensdugundji1951}, \cite{poppe-general}.
According to \cite{poppe-general}, we use the following convention:
Let $X$ and $Y$ be sets and $A\subseteq X$, $B\subseteq Y$; then
let be $(A,B):=\{f\in Y^X|\; f(A)\subseteq B\}$. Now let $X$ be a set, $(Y,\sigma)$ a topological space and ${\mathfrak A}\subseteq\Pm{X}$.
Then the topology $\tau_{{\mathfrak A}}$ on $Y^X$ (resp. $C(X,Y)$), which is defined by the open
subbase $\{(A,W)|\; A\in {\mathfrak A}, W\in \sigma\}$ is called the {\em
set--open topology, generated by ${\mathfrak A}$}, or shortly the {\em
${\mathfrak A}$--open topology}.\\

\noindent We know 

\begin{lem}[cf. \cite{nzjm}, lemma 3.4]\label{hatiso}
Let $(X,\tau),(Y,\sigma)$ be topological spaces,
let ${\mathfrak A}\subseteq\Pm{X}$ contain the singletons and ${\cal H}\subseteq Y^X$ be endowed with $\tau_{{\mathfrak A}}$.
Then the map
\begin{displaymath}
\mu_X:{\cal H}\to \Pm{Y}^{{\mathfrak A}}:\;f\to \mu_X(f):\forall A\in {\mathfrak A}: \mu_X(f)(A):= f(A), 
\end{displaymath}
is open, continuous and bijective onto its image, for $\Pm{Y}$ is equipped with Vietoris topology $\sigma_V$, and $\Pm{Y}^{{\mathfrak A}}$ with the generated pointwise topology.
\end{lem}

Now, the pointwise topology on $\Pm{Y}^{{\mathfrak A}}$ is just the product topology on $\prod_{A\in \ga}:\Pm{Y}_A$ (with all $\Pm{Y}_A$ being just copies of $\Pm{Y}$). By chosing $\ch:=C(X,Y)$, $\ga:=K(X)$ and consequently replacing $\Pm{Y}$ by $K(Y)$, we have the following situation:

\begin{align*}
\begin{xy}
\xymatrix{
C(X,Y) \ar[r]^{\!\!\!\!\!\!\!\!\!\mu_X}  & K(Y)^{K(X)}\;\;\; \cong &  \prod_{A\in K(X)}K(Y)_A \ar[d]^{\pi_A}\\
&&(K(Y),\sigma_V)\\
}
\end{xy}
\end{align*}
Of course, by $\pi_A$ we mean the canonical projection from the product to the factor $K(Y)_A=K(Y)$.

From lemma \ref{hatiso} we get the continuity of $\mu_X$, if $C(X,Y)$ is equipped with compact-open topology, thus in this case all compositions $\pi_A\circ \mu_X$ are continuous, too. 

Moreover, $\mu_X$ is even a homeomorphism onto its image and the product structure is initial w.r.t. the projections. So, the question arises, whether or not the Vietoris topology $\sigma_V$ on $K(Y)$ is final w.r.t. all $\pi_A\circ \mu_X$.


\begin{lem}\label{vietoris-final1}
Let $(X,\tau)$, $(Y,\sigma)$ be topological spaces and let $\sigma_V$ be the Vietoris topology on $K(Y)$.
Then for every  $\go\in\sigma_V$ and every $A\in K(X)$ the set $(\pi_A\circ \mu_X)^{-1}(\go)\subseteq C(X,Y)$ is open w.r.t. the compact-open topology.
\end{lem} 

\begin{proof}
Let $A\in K(X)$ be given and let $F\in Cl(Y)$ be a closed subset of $Y$. Then $(\pi_A\circ \mu_X)^{-1}(F^+)=\{f\in C(X,Y)|\; f(A)\subseteq Y\setminus F\}=(A,Y\setminus F)\in\tau_{co}$.
Let now $O\in\sigma$ be given, then

\noindent $(\pi_A\circ \mu_X)^{-1}(O^-)=\{f\in C(X,Y)|\; f(A)\cap O\neq\emptyset\}$ $=$ $\bigcup_{a\in A}(\{a\},O)$ $\in\tau_{co}$.

So, because the $F^+$ and $O^-$ form a subbase of $\sigma_V$, for $\go\in\sigma_V$ the preimage $(\pi_A\circ \mu_X)^{-1}(\go)$ is an element of $\tau_{co}$.
\end{proof}

\begin{cor}
Let $(Y,\sigma)$ be a topological space. For every topological space let $C(X,Y)$ be equipped with compact-open topology.

Then the Vietoris topology $\sigma_V$ on $K(Y)$ is contained in the final topology w.r.t. all $\pi_A\circ\mu_{_{(X,\tau)}}$, $(X,\tau)\in\cb$, $A\in K(X,\tau)$, for every class $\cb$ of topological spaces.
\end{cor}

\begin{prop}\label{halb-t4}
Let $(Y,\sigma)$ be a topological $T_3$-space, $K\subseteq Y$ compact and $O\subseteq Y$ open with $K\subseteq O$. Then an open set $U$ exists with $K\subseteq U\subseteq \ol{U}\subseteq O$. Especially,  $\ol{K}\subseteq O$ holds.
\end{prop}
\begin{proof}
$K\subseteq O$ just means $K\cap (Y\setminus O)=\emptyset$ and $(Y\setminus O)$ is closed. Thus, by $T_3$, for every element $k\in K$ there are $U_k, V_k\in \sigma$ s.t. $k\in U_k, Y\setminus O\subseteq V_k$ and $U_k\cap V_k=\emptyset$. The $U_k$'s cover $K$, so by compactness a finite subcover $U_{k_1}, ..., U_{k_n}$ exists. Let $U:=\bigcup_{i=1}^n$ and $V:=\bigcap_{i=1}^n$, so $U,V$ are open, $U\cap V=\emptyset$, $K\subseteq U$ and $Y\setminus O\subseteq V$ hold, i.e. 
\bdm
K\subseteq U\subseteq Y\setminus V\subseteq O\text{ .}
\edm

Now, $Y\setminus V$ is closed, so we get
\bdm
\ol{K}\subseteq \ol{U} \subseteq  \ol{Y\setminus V}=Y\setminus V\subseteq O\text{ .}
\edm
\end{proof}

\begin{lem}\label{vietoris-contains-quotient}
Let $(Y,\sigma)$ be an infinite\footnote{A finite Hausdorff (or even $T_1$-) space $Y$ would be discrete and the Vietoris hypertopology on $K(Y)=\Pm{Y}$ would become discrete, too. That is not such an exorbitant interesting case {\em here} for me.} Hausdorff $T_3$-space and let $(K(Y),\sigma_V)$ be its Vietoris Hyperspace of compact subsets. Let furthermore $\delta$ be the discrete topology
on $Y\times Y$ and denote by $(Z,\zeta)$ the Stone-\v{C}ech-com\-pac\-ti\-fi\-ca\-tion of $(Y\times Y,\delta)$.

Then $\sigma_V$ is the final topology on $K(Y)$ w.r.t. $\pi_Z\circ\mu_Z: C(Z,Y)\to K(Y)$, where $C(Z,Y)$ is endowed with compact-open topology $\tau_{co}$.
\end{lem}

\begin{proof}
From Lemma \ref{vietoris-final1} we know that $\sigma_V$ is contained in the final topology w.r.t. $\pi_Z\circ\mu_Z$, so we only have to show, that every open set of the final topology also belongs to $\sigma_V$. Let $\go$ be an open set of the final topology, i.e. $(\pi_Z\circ\mu_Z)^{-1}(\go)\in \tau_{co}$, and let $A\in\go$.\\

We want to show, that there exist finitely many open sets $U_1,..., U_m\in\sigma$ s.t. $A\in \vietoris{U_1,...,U_m}{K(Y)}\subseteq \go$.\\

At first, chose any surjection $s$ from $Y$ onto $A\subseteq Y$. Then extend it to a surjection $f_A:Y\times Y\to A$ by $f_A(y_1,y_2):=s(y_1)$, just meaning, that $f_A$ maps such pairs with equal first component to the same image.\\

Now, endowing $Y\times Y$ with discrete topology, we get $f_A$ being continuous. So, if $(Z,\zeta)$ denotes the Stone-\v{C}ech-compactification of the discrete $Y\times Y$, there exists a continuous extension $F_A:Z\to A$ of $f_A$.

Because $F_A$ is an extension of $f_A$, we have 
\begin{equation}\label{Fsaturiert}
\forall (a,b), (c,d)\in Y\times Y: a=c\implies F_A(a,b)=F_A(c,d)\text{ .}
\end{equation}

Because $\go$ is open in the final topology, there are finitely many compact subsets $K_1,...,K_n\in K(Z)$ and open subsets $O_1,...,O_n\in \sigma$ s.t. $F_A\in \bigcap_{i=1}^n(K_i,O_i)\subseteq (\pi_X\circ\mu_X)^{-1}(\go)$.\\

We will improve the sets $K_i$ and $O_i$ a little in an appropriate manner.\\

\begin{enumerate}\renewcommand{\labelenumi}{(\alph{enumi})}
\renewcommand{\theenumi}{(\alph{enumi})}

\item For each $K_i$ and every $k\in K_i$ there is an open neighbourhood $U_k$ of $k$, s.t. $F_A(U_k)\subseteq O_i$, because $F_A$ is continuous. Now, $\zeta$ has a base consisting of clopen sets $B$ of the form $B=\ol{B\cap(Y\times Y)}$. So, there exist always such a clopen $B_k\subseteq U_k$ with $k\in B_k$ and $F_A(B_k)\subseteq O_i$. The family of all $B_k, k\in K_i$ is an open cover of $K_i$ and consequently there is a finite subcover $\{B_{k_1}, ..., B_{k_l}\}$, by compactness of $K_i$. Now let 
\bdm
K_i':=\bigcup_{j=1}^l B_{k_j}
\edm
and observe, that $K_i'$ as a finite union of clopen sets is clopen again, hence it is compact and of the form $K_i'=\ol{K_i'\cap(Y\times Y)}$. Furthermore we have $K_i\subseteq K_i'$ and consequently
\bdm
F_A\in (K_i', O_i)\subseteq (K_i,O_i)\text{ .}
\edm

\item\label{liste2} We want to have our $K$'s saturated in the sense, that whenever $(a,b)\in K\cap (Y\times Y)$ holds, then $\{a\}\times Y\subseteq K$ also holds. So, let us define
\bdm
D_i:=\bigcup_{\substack{a\in Y, \exists b\in Y:\\ (a,b)\in K_i'\cap(Y\times Y)}}\{a\}\times Y
\edm
and then $K_i'':=\ol{D_i}$. From the continuity of $F_A$ follows 
\begin{equation}
F_A(K_i'')=F_A(\ol{D_i})\subseteq \ol{F_A(D_i)}
\end{equation}
and from \eqref{Fsaturiert} we get 
\begin{equation}
F_A(D_i)=F_A\left(K_i'\cap(Y\times Y)\right)\text{ .}
\end{equation}
Of course, $F_A\left(K_i'\cap(Y\times Y)\right)\subseteq F_A(K_i')$ and $F_A(K_i')$ is compact and fulfills $F_A(K_i')\subseteq O_i$, so by proposition \ref{halb-t4} we get from $(Y,\sigma)$ being $T_3$

\begin{equation}
F_A(K_i'')\subseteq \ol{F_A(D_i)}\subseteq \ol{F_A(K_i')}\subseteq O_i \text{ .}
\end{equation}
Note, that all $K_i''$ are compact and clopen again, by construction as a closure of a subset of $Y\times Y$ in the Stone-\v{C}ech-compactification $(Z,\zeta)$ of the discrete $Y\times Y$. Clearly, $K_i''\supseteq K_i'$ holds, yielding $(K_i'',O_i)\subseteq (K_i',O_i)$, thus
\begin{equation}
F_A\in \bigcap_{i=1}^n(K_i'',O_i) \subseteq \bigcap_{i=1}^n (K_i',O_i) \text{ .}
\end{equation}

\item\label{liste3} To cover $Z$ (resp. $A$) with our compact sets, we add $K_0'':=Z$ (resp. $O_0:=Y$) and find of course 
\bdm
F_A\in \bigcap_{i=1}^n(K_i'',O_i)\; =\; \bigcap_{i=0}^n(K_i'',O_i)\text{ .}
\edm

For each $z\in Z$ define 
\bdm
I(z):=\left\{ i\in\{0,...,n\}\left|\; z\in K_i''\right\}\right.
\edm
and then 
\begin{equation}\label{defcz}
C(z):=\bigcap_{i\in I(z)}K_i''\setminus\left( \bigcup_{j\in \{0,...,n\}\setminus I(z)}K_j''\right) \text{ }
\end{equation}
as well as 
\begin{equation}\label{defoz}
V(z):=\bigcap_{i\in I(z)}O_i \text{ .}
\end{equation}

Obviously for every $z\in Z$ we have 
\bdm
F_A(C(z))\subseteq F_A\left(\bigcap_{i\in I(z)}K_i''\right)\subseteq \bigcap_{i\in I(z)}O_i\; = V(z)
\edm
implying $F_A\in (C(z),V(z))$.

The family of all $C(z)$ covers $Z$, because every $z\in Z$ is contained at least in it's own $C(z)$. Observe, that different $C(z_1)$ and $C(z_2)$ are disjoint: if $y\in C(z_1)\cap C(z_2)$ exists, then $I(z_1)=I(y)=I(z_2)$ follows, implying $C(z_1)=C(z_2)$ by \eqref{defcz}. 

Obviously, there are only finitely many different sets $C(z), V(z)$, because they are uniquely determined by $I(z)$, which is a subset of $\{0,...,n\}$ and this set has just finitely many subsets. So, for simplicity, let us denote them by $C_1,...,C_m$ and $V_1,...,V_m$, respectively.

It is clear, that the $C_j$'s are clopen (thus compact) and saturated in the sense of paragraph \ref{liste2}, by construction \eqref{defcz} from just clopen saturated $K_i''$'s.\\ 

For $G\in \bigcap_{j=1}^m(C_j, V_j) =  \bigcap_{z\in Z}(C(z), V(z))$ we find 
\bdm
\begin{array}{rclcl}
\forall i\in\{0,...,n\}&:& \forall z\in K_i''&:& i\in I(z)\\
&&\implies &:& G(z)\in V(z)\subseteq O_i\\
\implies &:& G(K_i'')\subseteq O_i\text{ .} &&
\end{array}
\edm
Consequently, we have 
\begin{equation}
\label{Csbesser}
F_A\in\bigcap_{j=1}^m(C_j,V_j)\subseteq \bigcap_{i=0}^n(K_i'',O_i)
\end{equation}

\item\label{liste4} At last, let us chose for every $j=1,...,m$ an open set $U_j\in \sigma$ s.t. $F_A(C_j)\subseteq U_j\subseteq \ol{U_j}\subseteq V_j$ holds, as provided by proposition \ref{halb-t4}. Of course, we have then automatically $F_A\in (C_j, U_j)\subseteq (C_j,V_j)$.\\
So, because the $C_j$'s cover $Z$, the $F_A(C_j)$'s cover $A$, and so the $U_j$'s do.
\end{enumerate}

With these $U_j$, $j=1,...,m$ we show $A\in \vietoris{U_1,...,U_j}{K(Y)}\subseteq \go$.\\ 

$A\in \vietoris{U_1,...,U_m}{K(Y)}$ is clear, because the $U_j$'s cover $A$, as seen in paragraph \ref{liste4}, and $\emptyset\neq F_A(C_j)\subseteq A\cap U_j$ for all $j=1,...,m$.\\

Let 
\begin{equation}\label{BinV}
B\in \vietoris{U_1,...,U_j}{K(Y)}
\end{equation}
be given. 

Because every $C_j$ is nonempty clopen and saturated in the sense of paragraph \ref{liste2}, $C_j\cap (Y\times Y)$ has the cardinality of $Y$. So, there exists a surjection $t_j:C_j\cap (Y\times Y)\to U_j\cap B$ (the range is not empty by \eqref{BinV}).

Now, define 
\bdm
t:(Y\times Y)\to B:\; t(x,y):=t_j(x,y)\text{ for }(x,y)\in C_j 
\edm
This $t$ is well defined, because the $C_j$'s are pairwise disjoint and cover $Z$ by paragraph \ref{liste3}, and it is a surjection onto $B$, because the $U_j$'s cover $B$ by \eqref{BinV} and the $t_j$ are surjections onto $U_j\cap B$. Our $t$ is continuous w.r.t. the discrete topology on $Y\times Y$, so it extends to a continuous $T:Z\to B$. 

By construction we have for each $j\in\{1,...,m\}$ 
\begin{equation}
T\left(C_j\cap (Y\times Y)\right)\subseteq U_j\text{ ,}
\end{equation}
implying $T(C_j)=T\left(\ol{C_j\cap (Y\times Y)}\right)\subseteq \ol{T\left(C_j\cap (Y\times Y)\right)} \subseteq \ol{U_j}$ by continuity, thus $T(C_j)\subseteq V_j$ by choice of $U_j$ in paragraph \ref{liste4}.\\

We find $T\in \bigcap_{j=1}^m(C_j,V_j)\subseteq (\pi_Z\circ\mu_Z)^{-1}(\go)$, yielding $B=\pi_Z\circ\mu_Z(T)\in \go$. This works for every $B\in \vietoris{U_1,...,U_m}{K(Y)}$, thus we have indeed $\vietoris{U_1,...,U_m}{K(Y)}\subseteq \go$. Consequently, $\go$ is a union of Vietoris-open subsets of $K(Y)$, just meaning $\go\in\sigma_V$.
\end{proof}

\begin{remark}
\begin{enumerate}
\item Of course, $Y\times Y$ with discrete topology is homeomorphic to $Y$ with discrete topology for infinite $Y$. So, we used $Y\times Y$ here just for convenience concerning the description of the \glqq saturated\grqq\ subsets within the proof. Nevertheless, even for finite $Y$ this proof works fine, but wouldn't do so with $Y$ instead of $Y\times Y$.
\item In Lemma \ref{vietoris-contains-quotient}, we required $T_2$ for $Y$, because we use the classical version of the Stone-\v{C}ech theorem here. This can be weakened, because we need only the extendability of continuous functions to the compactification - and not the uniqueness of the extension.
\end{enumerate}
\end{remark}

\begin{cor}
Let $(Y,\sigma)$ be a Hausdorff $T_3$-space. For every topological space let $C(X,Y)$ be equipped with compact-open topology.
Let $\cb$ be a class of topological spaces, that contains the Stone-\v{C}ech-compactification of a discrete space with cardinality at least $card(Y)$.\\
Then the Vietoris topology $\sigma_V$ on $K(Y)$ is the final topology w.r.t. all $\pi_A\circ\mu_{_{(X,\tau)}}$, $(X,\tau)\in\cb$, $A\in K(X,\tau)$.
\end{cor}

\bibliographystyle{abbrvdin}
\bibliography{biblio}

\noindent
Ren\'e Bartsch\\ TU Darmstadt, Dept. of Math.\\ Schlossgartenstr. 7\\ D-64289 Darmstadt (Germany)\\
email: \href{mailto:rbartsch@mathematik.tu-darmstadt.de}{rbartsch\symbol{64}mathematik.tu-darmstadt.de}\\ 

\end{document}